\documentclass{article}
\usepackage{amsmath}
\usepackage{amssymb}
\usepackage{amsthm}
\usepackage{graphicx}
\usepackage{float}
\newtheorem{theorem}{Theorem}

\newtheorem{conjecture}{Conjecture} 
\newtheorem{counterexample}{Counterexample}
\newtheorem{definition}[theorem]{Definition} 
\newtheorem{proposition}[theorem]{Proposition}
\newtheorem{corollary}[theorem]{Corollary}
\newtheorem{example}[theorem]{Example}
\title{Generalizing $p$-goodness to ordered graphs}

\begin{document}

\maketitle

\begin{center}
\large
Jeremy F.~Alm\footnote{The first, third, and fourth authors were supported in part by NSF REU Grant \# 1757717} \\ 
Department of Mathematics\\
 Lamar University\\
 Beaumont, TX 77710\\
USA \\
{\tt alm.academic@gmail.com} \\
\ \\
Patrick Bahls\\
  Department of Mathematics \\ University of North Carolina, Asheville \\ Asheville, NC 28804\\
  \texttt{pbahls@unca.edu}\\ 
  
\bigskip
Kayla Coffey\\
Department of Mathematics\\
 
 Stephen F.~Austin State University\\
 Nacogdoches, TX 75962\\
USA \\
{\tt kcoffey621@gmail.com } \\
\ \\
Carolyn Langhoff\\
Department of Computer Science\\
 Lamar University\\
 Beaumont, TX 77710\\
USA \\
{\tt clanghoff@lamar.edu } \\

\end{center}

\section{Introduction}

Let $R(G,H)$ denote the ordinary 2-color Ramsey number, i.e., $R(G,H)$ is the minimum $N$ such that  every 2-coloring of the edges of $K_N$ contains a red copy of $G$ or a blue copy of $H$. Given a connected graph $G$ of order $n$, the Tur\'an construction shows that 
 \[
        R(G,K_p) \geq (n-1)(p-1)+1.
    \]
Call $G$ \emph{$p$-good} if equality holds, i.e., 
 \[
        R(G,K_p) = (n-1)(p-1)+1.
    \]
    
It follows from [1] that no connected graph with a cycle can be $p$-good for all $p$, and it is not hard to show that every tree is $p$-good for all $p$. In Section \ref{sec:trees}, we give a new proof of this fact, since the new proof will inform our work on \emph{ordered} Ramsey numbers, which we turn to in Section \ref{sec:ord}.  In Section \ref{sec:char}, we attempt to characterize the ordered trees that are order-$p$-good for all $p$. Finally, in Section \ref{sec:goodness} we return to ordinary Ramsey numbers and find, for several graphs with a cycle, the maximum $p$ for which the graph is $p$-good. We conclude with some open problems in Section \ref{sec:Qs}.

\section{A new proof of a theorem on trees}\label{sec:trees}

\begin{enumerate}
\item We say that the connected graph $H$, $|V(H)|=n$, satisfies the \textit{embedding condition} (Emb) if for all graphs $G$ such that $\delta(G) \geq n-1$, $H$ embeds in $G$.  
\item We say that the connected graph $H$, $|V(H)|=n$, satisfies the \textit{$p$-goodness  condition} (PG) if $R(H,K_p)=(n-1)(p-1)+1$ for all complete graphs $K_p$, $p \geq 3$, i.e., $H$ is $p$-good for all $p$.
\end{enumerate}



\begin{theorem} \label{EmbImpliesOR}
If $H$ satisfies (Emb) then it also satisfies (PG).
\end{theorem}

\begin{proof}
Suppose that a connected graph $H$, $|H|=n$, satisfies (Emb). Our note above implies we only need show $R(H,K_m) \leq (n-1)(m-1)+1$.

We proceed by induction on $n$, with base case $m=3$. Let $N=2(n-1)+1$ and color the edges of $K_N$ red and blue ($r$ and $b$). If there is a vertex $v \in V(K_N)$ such that $d_b(v) \geq n$ then any blue edge induced by $N_b(v)$ yields a blue $K_3$, so may assume all edges induced by $N_b(v)$ are red. This gives a red copy of $K_n$, which must contain a red copy of $H$. Thus we may assume that $\Delta_b(K_N) \leq n-1$, so $\delta_r(K_N) \geq 2(n-1)-n-1 = n-1$, so the red edges of $K_N$ give a graph which must contain a copy of $H$, since $H$ satisfies (Emb).

Assuming we've established the result for a given $p$, let us suppose $N=(n-1)(p+1-1)+1=p(n-1)+1$ and once again color $K_N$ red and blue. If $\delta_r(G) \geq n-1$, (Emb) again gives us a red copy of $H$. Thus we suppose $\delta_r(K_N) \leq n-2$, so that $\Delta_b(K_N) \geq p(n-1)-(n-2) = (n-1)(p-1)+1$. Thus for some vertex $v$, $d_b(v) \geq (n-1)(p-1)+1$ and by the induction hypothesis the blue neighborhood $N_b(v)$ induces a graph containing either a red $H$ or a blue $K_{p-1}$. In the former case we are done and in the latter case this copy of $K_{p-1}$, along with $v$, forms the blue $K_p$ needed, and we are done.
\end{proof}

Bohman and Keevash  \cite{Bohman} have the following result, which implies superlinearity of $R(C_\ell,K_p)$:

\begin{theorem}\label{lowerBound}
For fixed $\ell\geq 4$ and $p\to\infty$, we have $R(C_\ell,K_p) = \Omega\left(p^\frac{\ell-1}{\ell-2} / \log p\right)$.
\end{theorem}

\begin{corollary}
Let $G$ be connected of order $n$. Then if $G$ contains a cycle, $R(G,K_p)$ grows faster than any function that is linear in $p$. Hence if $R(G,K_p)$ is linear in $p$, $G$ is a tree.  Therefore (PG) implies that $G$ is a tree.
\end{corollary}

\begin{theorem} \label{TreesSatisfyEmb}
Let $H$ be a simple connected graph. Then $H$ satisfies (Emb) if and only if $H$ is a tree.
\end{theorem}

\begin{proof}
Let $|H|=n$ and first suppose $H$ is not a tree, and therefore contains a cycle. Let $c$ denote the maximum length of a cycle in $H$. By a probabilistic argument (see \cite{Bollobas}, for example) there exists a graph $G$ such that $\delta(G) \geq n-1$ and with girth at least $c+1$. $H$ cannot embed in such a graph, showing that (Emb) does not hold.

Now suppose $H$ is a tree and that $G$ satisfies $\delta(G) \geq n-1$. We prove a stronger condition than (Emb), namely that given $u \in V(H)$ and $v \in V(G)$ there exists an embedding of $H$ into $G$ such that $u \mapsto v$. We prove this by induction, the base case $n=2$ being trivial. Assume the result for a given $n$ and let $|H|=n+1$ and $G$ such that $\delta(G) \geq n = (n+1)-1$. Let $H' = H \setminus \{w\}$ for some leaf $w \neq u$. Let $N(w) = \{w'\}$. By inductive hypothesis we may embed $H'$ in $G$, taking $u$ to $v$ and $w'$ to $x$ for some $x \in V(G)$. Since $d(x) \geq n$ and $|N(w') \setminus \{w\}| \leq n-1$, after embedding $H'$ in $G$ at least one vertex remains in $N(x)$ to which we may map $w$, finishing our embedding and our proof.
\end{proof}

Together our results give us a proof of an already-known fact, though it seems not to have been stated as such in the literature:\
\begin{theorem} \label{thm:equiv}
Let $G$ be a connected graph of order $n$. Then the following are equivalent:
\begin{enumerate}
    \item $G$ is a tree.
    \item $G$ satisfies (Emb).
    \item $G$ satisfies (PG).
\end{enumerate}

In particular, the collection of connected graphs that are $p$-good for all $p$ is precisely the collection of trees.
\end{theorem}


\section{Ordered Ramsey numbers}\label{sec:ord}

\begin{definition}
Let $G$ be a connected graph whose vertices are linearly ordered. Let $K_N$ have vertex set $\{0,1,\dots, n-1\}$. Let $r_<(G,K_p)$ denote the least $N$ such that if the edges of $K_N$ are colored in red or blue, either
\begin{enumerate}
    \item there is a red ordered copy of $G$, i.e., there is an order-preserving injection from $G$ to the red subgraph of $K_N$, or
    \item there is a blue copy of $K_p$.
\end{enumerate}
\end{definition}
Ordered Ramsey numbers are fairly new. The first author independently (re)discovered them while programming -- he was too lazy to code up all the ways to make a cycle out a given set of four vertices, so he wrote code to forbid only one such way, hence computing the ordered Ramsey number $r_<(C^<_4, C^<_4)$, where $C^<_4$ has the ``monotone'' ordering 0-1-2-3-0.
\begin{example}
While $R(C_4, K_4) = 10$, making $C_4$ 4-good, if $C_4^<$ is given the ``monotone'' ordering 0-1-2-3-0, then $r_<(C_4^<,K_4) = 14$. This ordered Ramsey number was computed via the SAT solver MiniSat. 
\end{example}

Now let us consider $p$-goodness for ordered graphs.

\begin{definition}
Let $G^<$ be an ordered graph on $n$ vertices. Then $G^<$ is called order-$p$-good if $r_<(G^<,K_p) = (n-1)(p-1)+1$.
\end{definition}

\begin{example}
Let $P_4^<$ be the monotone path 0-1-2-3.  Then $P_4^<$ is order-$p$-good for all $p$. 
\end{example}
In fact, all monotone paths are order-$p$-good for all $p$ (\cite{balko}, although it appears to be folklore).

\begin{example}
While all trees are $p$-good for all $p$, not all ordered trees are order-$p$-good for all $p$. For example, if $T^<$ is a path with ordering 0-3-1-2, then $r_<(T^<,K_4) = 11$, so $T^<$ is not even 4-good. 
\end{example}

We can, however, prove that a certain class of ordered graphs are order-$p$-good for all $p$. 

\begin{theorem}\label{thm:stars}
Let $S_n^<$ denote an $n$-vertex star whose highest vertex is its center. Then $r_<(S_n^<, K_p) = (n-1)(p-1)+1$. Hence $S_n^<$ is order-$p$-good for all $p$. 
\end{theorem}

\begin{proof}
We proceed by induction on $n$, the base case $n=2$ being trivial.

Suppose $p\geq 2$ and assume $r_<(S_{n-1}^<, K_p) = (n-2)(p-1)+1$. Consider an edge-coloring of $K_N$, $N=(n-1)(p-1)+1$, with no blue $K_p$. Let $X_0$ be the set of $(n-2)(p-1)$ lowest vertices, and let $x_0<x_1<\dots <x_{p-1}$ be the $p$ remaining highest vertices. Let $X_1=X_0 \cup \{x_0\}$. Since $|X_1|=(n-2)(p-1)+1$, the subgraph of $K_N$ induced by $X_1$ contains a red copy of $S_{n-1}^<$. Let $^1S_{n-1}^<$ be the copy with \emph{lowest center} (if there are multiple copies), and let $c_1$ be its center. 

Now  let $X_2=(X_1\setminus \{c_1\}) \cup \{x_1\}$. Again, since $|X_2|=(n-2)(p-1)+1$, the subgraph of $K_N$ induced by $X_2$ contains a red copy of $S_{n-1}^<$. Let $^2S_{n-1}^<$ be the copy with \emph{lowest center} (if there are multiple copies), and let $c_2$ be its center. Note that $c_2 > c_1$, since if $c_2 < c_1$, $c_2$ would have been chosen in the previous stage. 

Proceed in this fashion, setting $X_{k+1}=(X_k\setminus \{c_k\}) \cup \{x_k\}$, getting at total of $p$ copies $^kS_{n_1}^<$ with centers $c_k$, and $c_1<c_2< \dots < c_p$. Consider the subgraph induced by $\{c_1, \dots, c_p\}$. Since there is no blue $K_p$, there is some red edge $c_ic_j$, $i<j$, in this subgraph.  But then $^jS_{n-1}^<$, adjoined with $c_i$ and edge $c_ic_j$, is a red $S_n^<$. 

\end{proof}





\section{Characterizing ordered trees that are order-$p$-good for all $p$}\label{sec:char}
In this section, we eliminate several tempting-looking conjectural characterizations of trees that are order-$p$-good for all $p$.

Recall the characterization from Theorem \ref{thm:equiv}: the 
only connected graphs that are $p$-good for all $p$ are those that satisfy (Emb). This suggests the following conjecture:

\begin{conjecture}
    An ordered tree $T$ is order-$p$-good for all $p$ if and only if

$(C0)$: $T$ order-embeds in EVERY connected ordered graph with minimum degree $n-1$.
\end{conjecture}


\begin{counterexample} 
The monotone path $P_4^<$ of order 4 does not order-embed in the following graph. 
    \begin{figure}[H]
        \centering
        \includegraphics[width=1in]{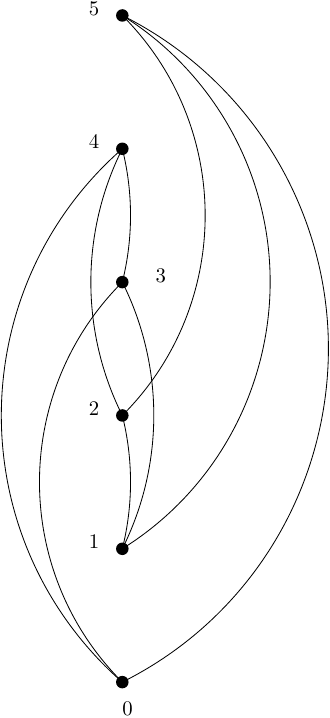}
        \caption{The monotone path on 4 vertices does not embed}
       
    \end{figure}
\end{counterexample}


\begin{conjecture}
    An ordered tree $T$ is order-$p$-good for all $p$ if and only if\\

    $(C1)$: there exists a vertex $v$ such that every subpath of $T$ ending at $v$ is monotone.

\end{conjecture}


\begin{counterexample}
The following tree satisfies (C1) but is not 4-good. 
    \begin{figure}[H]
        \centering
        \includegraphics[width=0.8in]{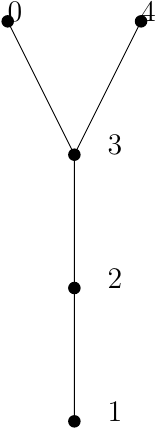}
        \caption{$T$ satisfies $(C1)$ but $r_<(T,K_4) = 16$, not 13}
       
    \end{figure}
\end{counterexample}


\begin{conjecture} 
    An ordered tree $T$ is order-$p$-good for all $p$ if and only if

    $(C2)$: Every subpath of $T$ is monotone.
\end{conjecture}


\begin{counterexample} 

The stars in Theorem \ref{thm:stars} serve as counterexamples to (C2).
    \begin{figure}[H]
        \centering
        \includegraphics[width=2in]{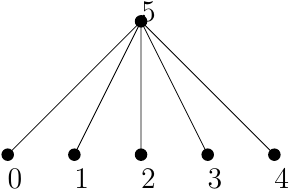}
        \caption{$(C2)$ is too strong. Every star with center its highest vertex is $p$-good for all $p$ but does NOT satisfy $(C2)$.}
       
    \end{figure}
\end{counterexample}


\section{Goodness of some small graphs}\label{sec:goodness}

Now we turn our attention back to unordered graphs. 
\begin{definition}
Let $H$ be a connected graph \emph{not} satisfying (PG), i.e.~$H$ is not a tree.  Then the \emph{goodness} of $H$ is the maximum $p$ such that $H$ is $p$-good.
\end{definition}

In this section,  let $H_1$ stand for $K_3$ with a pendant edge; let $H_2$ stand for $H_1$ with a pendant edge; let $H_3$ stand for $C_4$ with a pendant edge; let $H_4$ stand for $K_4$ with a pendant edge; let $H_5$ stand for $H_4$ with a pendant edge; let $H_6$ stand for $K_4 - e$ with a pendant edge; let $H_7$ stand for $H_6$ with a pendant edge.

\begin{proposition} \label{goodH1}
The goodness of $H_1$ is 4.
\end{proposition}

\begin{proof}
We may prove that $R(H_1,K_3)=7$.  This is a bit of case analysis, but the proof is straight-forward so we omit it.  It also follows from the following equality from \cite{BE}:

\begin{quote}
If $G$ is a connected graph on $n-1$ vertices, and $G_1$ is formed by adding to $G$ a pendant edge, then 
\end{quote}
\begin{equation}\label{eq:2}
R(G_1,K_p)=\max\{ R(G,K_p),R(G_1,K_{p-1})+n-1\}.
\end{equation}

Setting $p=4$ and $H_1=G$, we get
\begin{align*}
R(H_1,K_4) &=\max\{ R(3,4),R(H_1,K_3)+3\}\\
&= \max\{ 9,7+3\}\\
&= 10.
\end{align*}

Now $10=3\cdot 3+1=(|H_1|-1)(p-1)+1$, so $H_1$ is 4-good.  However, $H_1$ fails to be 5-good.  To see this, let $p=5$.  Then 
\begin{align*}
R(H_1,K_5) &=\max\{ R(3,5),R(H_1,K_4)+3\}\\
&= 14,
\end{align*}

which is one more than would be the case if $H_1$ were 5-good.

Therefore the goodness of $H_1$ is 4.

\end{proof}

In  Table \ref{Table} below, we use equation \eqref{eq:2} along with known Ramsey numbers from \cite{Rad} to compute the goodness of several graphs  on 4 or 5 vertices having pendant edges. The numbers in parentheses refer to what the Ramsey number would need to be for that  graph to be $p$-good for that particular value of $p$. The arguments are similar to those used in  the proof of Proposition \ref{goodH1}.

\begin{table}[h!b]
\centering
\begin{tabular}{|c|c|c|c|c|c|c|}
\hline & $K_3$ & $H_1$ & $H_2$ & $C_4$ & $H_3$ & $K_4$   \\ \hline
$R(\underline{\phantom{H}}, K_3 )$ & 6 & 7 & 9 & 7 & 9 & 9  \\ \hline 
$R(\underline{\phantom{H}}, K_4 )$ & 9 & 10 & 13 & 10 & 10 & 18  \\ \hline 
$R(\underline{\phantom{H}}, K_5 )$ & 14 & 14 (13) & 17 & 14 & 14 & 25 \\ \hline 
$R(\underline{\phantom{H}}, K_6 )$ & 18 & 18 & 21 & 18 & 18 & 35-41  \\ \hline 
$R(\underline{\phantom{H}}, K_7 )$ & 23 & 23 & 25 & 22 & 22 &   \\ \hline 
$R(\underline{\phantom{H}}, K_8 )$ & 28 & 28 & 29 & 26 & 26 &   \\ \hline 
$R(\underline{\phantom{H}}, K_9 )$ & 36 & 36 & 36 (33) & 30-32 & 32 &   \\ \hline 
$R(\underline{\phantom{H}}, K_{10} )$ & 40-43  & 34-39 & $\ge 36$  &  &  &   \\ \hline 
goodness & 2 & 4 & 8 & 4 & $\ge 9$ & 2  \\ \hline
\end{tabular}
\begin{tabular}{|c|c|c|c|c|c|}
\hline &  $H_4$ & $H_5$ & $K_4 - e$ & $H_6$ & $H_7$  \\ \hline
$R(\underline{\phantom{H}}, K_3 )$ &  9 & 11 & 7 & 9 & 11 \\ \hline 
$R(\underline{\phantom{H}}, K_4 )$ &  18 (13) & 18 (16) & 11 & 13 & 16  \\ \hline 
$R(\underline{\phantom{H}}, K_5 )$  & 25 &  & 16 & 17 & 21\\ \hline 
$R(\underline{\phantom{H}}, K_6 )$  &  &  & 21 & 21 & 26 \\ \hline 
$R(\underline{\phantom{H}}, K_7 )$  &  &    & 28-31 & $\ge 28$ (25)  & 31 \\ \hline 
$R(\underline{\phantom{H}}, K_8 )$  &  &  &  &    & \\ \hline 
$R(\underline{\phantom{H}}, K_9 )$  &  &  &  &    & \\ \hline 
$R(\underline{\phantom{H}}, K_{10} )$ &  &  &  &    & \\ \hline 
goodness  & 3 & 3 & 3 & 6 & $\ge 7$ \\ \hline
\end{tabular}
\caption{Goodness of some small graphs}
\label{Table}
\end{table}

\section{Open Questions}\label{sec:Qs}

\begin{enumerate}
    
    \item Characterize the ordered trees that are order-$p$-good for all $p$.
    
    Failing that,

    \item Characterize the ordered paths that are order-$p$-good for all $p$.

    \item Find a sufficient (or necessary!) condition for $r_<(G^<, K_p) > R(G, K_p)$.
\end{enumerate}

\end{document}